\documentclass[11pt]{article}

\usepackage{amsmath,amssymb,amsthm,color}
\usepackage{booktabs}
\usepackage{algorithmic}
\usepackage{algorithm}

\newcommand{\mean}{{E}}
\newcommand{\prob}{\mathbb{P}}
\newcommand{\corr}{\text{corr}}

\newcommand{\unif}{\textsf{Unif}}

\newcommand{\bern}{\textsf{Bern}}
\newcommand{\ind}{\textbf{1}}
\newcommand{\stddev}{\textrm{sd}}
\newcommand{\probability}{P}

\begin{document}

\newtheorem{theorem}{Theorem}
\newtheorem{lemma}{Lemma}
\newtheorem{proposition}{Proposition}
\theoremstyle{definition}
\newtheorem{definition}{Definition}

\title{Multivariate distributions with fixed
marginals and correlations}

\author{Mark Huber\thanks{Claremont McKenna College (email: {\tt mhuber@cmc.edu})} ~and Nevena Mari\'c\thanks{University of Missouri-St. Louis (email: {\tt maric@math.umsl.edu})} }

\maketitle

\begin{abstract}
Consider the problem of drawing random variates 
$(X_1,\ldots,X_n)$ from a distribution where the marginal of each
$X_i$ is specified, as well as the correlation between every pair $X_i$ and 
$X_j$.  For given marginals, the Fr\'echet-Hoeffding bounds put a lower and 
upper bound on the correlation between $X_i$ and $X_j$.  
Any achievable correlation between 
$X_i$ and $X_j$ is a convex combinations of these bounds. 
The value $\lambda(X_i,X_j) \in [0,1]$ of this convex 
combination is called here the convexity parameter
of $(X_i,X_j),$ with $\lambda(X_i,X_j) = 1$ 
corresponding to the upper bound and maximal correlation.
For given marginal distributions functions 
$F_1,\ldots,F_n$ of $(X_1,\ldots,X_n)$ we show that 
$\lambda(X_i,X_j) = \lambda_{ij}$
if and only if there exist symmetric Bernoulli random variables 
$(B_1,\ldots,B_n)$ (that is $\{0,1\}$ random variables with mean 1/2) 
such that 
$\lambda(B_i,B_j) = \lambda_{ij}$. 
In addition, we characterize completely the set
of convexity parameters for symmetric Bernoulli marginals in 
two, three and four
dimensions.

\end{abstract}

\section{Introduction}
Consider the problem of simulating 
a random vector 
$(X_1,\ldots,X_n)$ with second moments 
where for all $i$ the cumulative 
distribution function (cdf) of $X_i$ is $F_i$, and for all $i$ and $j$
the correlation between
$X_i$ and $X_j$ should be $\rho_{ij} \in [-1,1]$.  
The correlation here is the 
usual notion
\[
\corr(X,Y) = \frac{\mean[(X - \mean(X))(Y - \mean(Y))]}{\stddev(X)\stddev(Y)}
 = \frac{\mean[XY] - \mean[X]\mean[Y]}{\stddev(X)\stddev(Y)},
\]
for standard deviations $\stddev(X)$ and $\stddev(Y)$ that are finite.

Let $\Omega$ denote the set of matrices with
entries in $[-1,1]$, all the diagonal entries equal $1$, and 
are nonnegative definite.  Then it is well known that any correlation
matrix $(\rho_{ij})$ must lie in $\Omega$.

This problem, in different guises, 
appears in numerous fields: physics~\cite{smith1981gust}, 
engineering~\cite{lampard1968stochastic}, ecology~\cite{Dias2008}, 
and finance~\cite{Lawrance1981}, to name just a few. 
Due to its applicability in the generation of 
synthetic optimization problems, it has also received special attention 
by the simulation community ~\cite{hill1994},~\cite{henderson2000}.

A variety of approaches exist for this well studied problem.
When the marginals are normal and the distribution is continuous with
respect to Lebesgue measure, this is just the problem of generating
a multivariate normal with specified correlation matrix.  It is well 
known how to accomplish this (see, for instance~\cite{fishman1996}, p. 223)
for any matrix in $\Omega$.

For marginals that are not normal, the question is very much harder.
A common method is to employ families of copulas 
(see for instance~\cite{nelson1999}), but there are very few techniques
that apply to general marginals.  Instead, different families of copulas
typically focus on different marginal distributions.

Devroye and Letac~\cite{devroyel2010} showed that if the marginals are
beta distributed with equal parameters at least $1/2$, then when the
dimension is three it is possible to simulate such a vector where
the correlation is any matrix in $\Omega$.
This set of beta distributions 
includes the important case of uniform $[0,1]$ marginals, but
they have not been able to extend their technique to higher dimensions.

Chaganty and Joe~\cite{chagantyj2006} characterized the achievable
correlation matrices when the marginals are Bernoulli.  When the dimension
is 3 their characterization is easily checkable, in higher dimensions they
give a number of inequalities that grows exponentially in the dimension. 

For the case of general marginals,  
in statistics there is a tradition of using transformations  of 
mutivariate normal vectors dating back to Mardia~\cite{mardia1970} 
and Li and Hammond~\cite{li1975}. This approach relies heavily
on developing usable numerical methods.
In this paper we approach the same problem using exclusively 
probabilistic techniques. 

We show that for many correlation matrices
the problem of simulating from a multivariate distribution with fixed marginals
and specified correlation 
can be reduced to showing the existence of a multivariate
distribution whose marginals are Bernoulli with mean $1/2$, and for
each pair of marginals, there is a specified probability that the pair takes
on the same value.
For $n=2,3,4$ we are able to give necessary and sufficient 
conditions on those agreement probabilities in order for such a distribution
to exist.

\paragraph{The convexity graph.}  Any two random variables
$X$ and $Y$ have correlation in $[-1,1]$, but if the 
marginal distributions of $X$
and $Y$ are fixed, it is generally not possible to build a bivariate
distribution for any correlation in $[-1,1]$.  For instance,
for $X$ and $Y$ both exponentially distributed, the correlation must lie
in $[1 - \pi^2/6,1]$.  The range of achievable correlations is always 
a closed interval.

For two dimensions it is well 
known how to find the minimum and maximum correlation.
  These come from the inverse transform method, which works as
follows.  First, given a cdf $F$, define the pseudoinverse of the cdf as
\begin{equation}
F^{-1} = \inf\{x:F(x) \geq u\}.
\end{equation}
When $U$ is uniform over the interval $[0,1]$ (write $U \sim \unif([0,1])$),
$F^{-1}(U)$ is a random variable with cdf $F$ 
(see for instance p.~28 of~\cite{devroye1986}).  Since $U$ and $1 - U$ 
have the same distribution, both can be used in the inverse transform
method.  The random variables $U$ and $1 - U$ are {\em antithetic} random
variables.  Of course $\corr(U,U) = 1$ and $\corr(U,1-U) = -1$, so these 
represent an easy way to get minimum and maximum correlation when the 
marginals are uniform random variables.

The following theorem comes from work of Fr\'echet~\cite{frechet1951}
and Hoeffding~\cite{hoeffding1940}.

\begin{theorem}[Fr\'echet-Hoeffding bound]
\label{THM:fhbound}
For $X_1$ with cdf $F_1$ and $X_2$ with cdf $F_2$,
and $U \sim \unif([0,1])$:
\[
\corr(F_1^{-1}(U),F_2^{-1}(1 - U)) \leq \corr(X_1,X_2) \leq 
  \corr(F_1^{-1}(U),F_2^{-1}(U)).
\]
\end{theorem}

In other words, the maximum correlation between $X_1$ and $X_2$ is 
achieved when the same uniform is used in the inverse transform method
to generate both.  The minimum correlation between $X_1$ and $X_2$ is
achieved when antithetic random variates are used in the inverse transform
method. 

\begin{definition}
Consider random variables $X$ and $Y$ with finite second moments, and
cdf $F_X$ and $F_Y$ respectively.  For $U$ uniform on $[0,1]$,
let $\rho^- = \corr(F_X^{-1}(U)F_Y^{-1}(1 - U))$ and 
$\rho^+ = \corr(F_X^{-1}(U) F_Y^{-1}(U))$.  Then (by the Fr\'echet-Hoeffding
bound) there is a unique $\lambda \in [0,1]$ such that 
\[
\corr(X,Y) = \lambda \rho^{+} + (1 - \lambda) \rho^{-}.
\]
Call $\lambda = \lambda(X,Y)$ the {\em convexity parameter} of 
$X$ and $Y$.
\end{definition}

\begin{definition}
Consider $(X_1,\ldots,X_n)$ with finite second moments, 
where each $X_i$ has cdf $F_i$, and the correlation between 
$X_i$ and $X_j$ is $\rho_{ij}$.  Then 
the complete graph on $\{1,\ldots,n\}$ where edge $\{i,j\}$ has weight
$\lambda_{ij} = \lambda(X_i,X_j)$ is the {\em convexity graph}
of the distribution.
\end{definition}

Let ${\cal B}_n$ be the set of probabilities on $\{0,1\}^n$ such that if
$(B_1,\ldots,B_n) \sim \mu$ where $\mu \in {\cal B}_n$, then
$\probability(B_i = 1) = 1/2$ for all $i$.

\begin{theorem}
\label{THM:main}
Let $(B_1,\ldots,B_n) \sim \mu \in {\cal B}_n$.  Then
$\lambda(B_i,B_j) = \probability(B_i = B_j)$ for all $i < j$.
For all distribution functions $F_1,\ldots,F_n$ with second moments,
there exists a distribution for $(X_1,\ldots,X_n)$ such that for all
$i$ we have $\probability(X_i \leq x) = F_i(x)$ and for all
$i < j$ we have $\lambda(X_i,X_j) = \lambda(B_i,B_j)$.  
\end{theorem}

\begin{proof} For $(B_i,B_j)$ with symmetric Bernoulli marginals,
the value of 
either $\probability(B_i = B_j)$ or $\corr(B_i,B_j)$ (which is one-to-one with
$\lambda(B_i,B_j)$) completely 
determines the bivariate distribution.  It is then
straightforward to verify that $\probability(B_i = B_j) = \lambda(B_i,B_j)$.

Next, consider $U$ uniform on $[0,1]$ independent of
$(B_1,\ldots,B_n)$.  Then
$X_i = F_i^{-1}(U B_i + (1 - U)(1 - B_i))$ has 
the correct marginals and again it is straightforward to show
$\lambda(X_i,X_j) = \lambda(B_i,B_j)$.
\end{proof}

Theorem~\ref{THM:main} immediately gives us a way to simulate from
a distribution $(X_1,\ldots,X_n)$ with given convexity parameters in
linear time, provided it is possible to simulate from a multivariate
symmetric Bernoulli with the same convexity parameters.
The next result characterizes when such a multivariate Bernoulli exists
in two, three, and four dimensions, 
and gives necessary conditions for higher dimensions.

\begin{theorem}
\label{THM:nec}
Suppose $(B_1,B_2,\ldots,B_n)$ are random variables with
$\prob(B_i = 1) = \prob(B_i = 0) = 1/2$ for all $i$.  When $n = 2$, it
is possible to simulate $(B_1,B_2)$ for any $\lambda_{12} \in [0,1]$.
When $n = 3$, it is possible to simulate $(B_1,B_2,B_3)$ if and only if
\[
1 + 2\min\{\lambda_{23},\lambda_{12},\lambda_{13}\}
 \geq \lambda_{23} + \lambda_{12} + \lambda_{13} \geq 1.
\]
When $n = 4$, it is possible to simulate $(B_1,B_2,B_3,B_4)$ if and only if
\[
\ell \leq u + 1 \text{ and } 1 \leq u
\]
where
\begin{align*}
\ell &= \max (\lambda_{14} + \lambda_{14} + \lambda_{13} + \lambda_{23},
 \lambda_{14} + \lambda_{34} + \lambda_{12} + \lambda_{23},
 \lambda_{24} + \lambda_{34} + \lambda_{12} + \lambda_{13}) \\
u &= \min_{\{i,j,k\}} \lambda_{ij} + \lambda_{jk} + \lambda_{ik}.
\end{align*}
\end{theorem}

The rest of the paper is organized as follows.  In the next
section, Theorem~\ref{THM:nec} is shown.  In
Section~\ref{SEC:asymmetric}, 
the set of multivariate asymmetric Bernoulli distributions is linked
to that of symmetric Bernoulli distributions.

\section{Proof of Theorem~\ref{THM:nec}}
\label{SEC:nec}

\subsection{The $n = 2$ and $n = 3$ cases}

\begin{lemma}
\label{LEM:two}
For any $\lambda_{12} \in [0,1]$,
there exists a unique joint distribution on $\{0,1\}^2$ such that 
$(B_1,B_2)$ with this distribution has $B_1,B_2 \sim \bern(1/2)$ and 
$\probability(B_1 = B_2) = \lambda_{12}$.
\end{lemma}

\begin{proof}
Let $p_{ij} = \probability(B_1 = i,B_2 = j)$.  Then the equations that are
necessary and sufficient to meet the distribution and convexity 
conditions are:
\[
p_{10} + p_{11} = 0.5,\ p_{01} + p_{11} = 0.5,\ p_{11} + p_{00} = \lambda_{12},
 \text{ and } p_{00} + p_{01} + p_{10} + p_{11} = 1.
\]
This system of linear equations has full rank, so there exists a 
unique solution.  Given there is a unique solution, it is easy to 
verify that solution is:
\[
p_{00} = (1/2)\lambda_{12},\ p_{01} = (1/2)[1-\lambda_{12}],
  \ p_{10} = (1/2)[1-\lambda_{12}],\ p_{11} = (1/2)\lambda_{12}.
\]
\end{proof}

This provides an alternate algorithm to that
found in~\cite{dukicm2013} for simulating from bivariate distributions with
correlation between $\rho^-_{1,2}$ and $\rho^+_{1,2}$. 

\begin{lemma}
\label{LEM:3B}
A random vector $(B_1,B_2,B_3)$ with $B_i \sim \bern(1/2)$ exists
(and is possible to simulate from in a constant number of steps) 
if and only if the concurrence graph
satisfies
\[
1 \leq \lambda_{23} + \lambda_{12} + \lambda_{13} \leq 
  1 + 2 \min\{\lambda_{12},\lambda_{13},\lambda_{23}\}
\]
\end{lemma}

\begin{proof}
Let $p_{ijk} = \probability(B_1 = i,B_2 = j,B_3 = k)$.  The first
condition is $\sum_{i,j,k} p_{i,j,k} = 1.$  There are 
three conditions from the marginals:
\[
\sum_{j,k \in \{0,1\}} p_{1jk} = 0.5,\ 
\sum_{i,k \in \{0,1\}} p_{i1k} = 0.5,\ 
\sum_{ij \in \{0,1\}} p_{ij1} = 0.5,
\]
and three conditions from the correlations
\[
\sum_{k \in \{0,1\}} p_{00k} + p_{11k} = \lambda_{12},\ 
\sum_{j \in \{0,1\}} p_{0j0} + p_{1j1} = \lambda_{13},\ 
\sum_{i \in \{0,1\}} p_{i00} + p_{i11} = \lambda_{23}.
\]
To get 8 equations, suppose that $p_{111} = \alpha$.

This 8 by 8 system of equations has full rank, so there is a unique
solution.  It is easy to verify that the solution is 
\begin{align*}
p_{000} &= (1/2)(\lambda_{12} + \lambda_{13} + \lambda_{23} - 1) - \alpha 
  \hspace*{1em}
  & p_{100} &= (1/2)(1-(\lambda_{12} + \lambda_{13})) + \alpha \\
p_{001} &= (1/2)(1-(\lambda_{13} + \lambda_{23})) + \alpha 
  & p_{101} &= (1/2)\lambda_{13} - \alpha \\
p_{010} &= (1/2)(1-(\lambda_{12} + \lambda_{23})) + \alpha 
  & p_{110} &= (1/2)\lambda_{12} - \alpha \\
p_{011} &= (1/2)\lambda_{23} - \alpha 
  & p_{111} &= \alpha 
\end{align*}

In order for this solution to yield probabilities, all must lie in $[0,1]$.
Since $p_{111} = \alpha$, $\alpha \geq 0$.
The $p_{011},p_{101},$ and $p_{110}$ equations then give
\begin{equation}
\label{EQN:restriction1}
0 \leq \alpha \leq (1/2)\min\{\lambda_{12},\lambda_{23},\lambda_{13}\}.
\end{equation}

The $p_{000}$ equation requires that 
\begin{equation}
\label{EQN:restriction2}
\alpha \leq (1/2)(\lambda_{13} + \lambda_{12} + \lambda_{23}  - 1).
\end{equation}
With these two conditions, 
equations $p_{001}$, $p_{010}$, and $p_{100}$ give the constraint
\begin{equation}
\label{EQN:restriction3}
(1/2)(\lambda_{13} + \lambda_{12} + \lambda_{23} - 
  \min\{\lambda_{13},\lambda_{12},\lambda_{23}\} - 1) \leq \alpha.
\end{equation}

Combining~\eqref{EQN:restriction3} and 
\eqref{EQN:restriction1} gives
\begin{equation}
\label{EQN:restriction4}
(1/2)(\lambda_{13} + \lambda_{12} + \lambda_{23} - 
  \min\{\lambda_{13},\lambda_{12},\lambda_{23}\} - 1) \leq 
   (1/2) \min\{\lambda_{13},\lambda_{12},\lambda_{23}\}.
\end{equation}

As long an $\alpha \geq 0$ exists satisfying \eqref{EQN:restriction2}
 and~\eqref{EQN:restriction4} holds, there exists a solution.
\end{proof}

From this result we see 
that not all positive definite correlation matrices are attainable
with $\bern(1/2)$ marginals.  For instance, if
$\lambda_{12} = \lambda_{13} = \lambda_{23} = 0.3$, then
$\rho_{12} = \rho_{13} = \rho_{23} = -0.4.$  With diagonal entries 
$1$, the $\rho$ values form a positive definite graph, but it is impossible
to build a multivariate distribution with $\bern(1/2)$ marginals with
these correlations.

\subsection{The $n = 4$ case:  asymmetric Bernoulli distributions}
\label{SEC:asymmetric}

To show the $n = 4$ case, it will be useful to understand
the problem of drawing a multivariate Bernoulli 
$(X_1,\ldots,X_n)$ where $X_i \sim \bern(p_i)$ where $i$ is not necessarily
$1/2$.

\begin{lemma}
\label{LEM:asymmetric}
An $n$ dimensional 
multivariate Bernoulli distribution where the marginal of component
$i$ is $\bern(p_i)$ and concurrence graph $\Lambda$ exists if and only
if an $n + 1$ dimensional multivariate Bernoulli distribution exists with
$\bern(1/2)$ marginals and concurrence graph
\[ \left(
\begin{array}{c|c}
\raisebox{-15pt}{{\Large\mbox{{$\Lambda$}}}} & p_1 \\[-4ex]
 & \vdots \\
 & p_n \\
\hline
p_1 \cdots p_n & 1 \\ 
\end{array}
\right)
\]
\end{lemma}

\begin{proof}
Suppose such an $n + 1$ dimensional distribution exists with $\bern(1/2)$
marginals and specified concurrence graph.  Let $(B_1,\ldots,B_{n+1})$ be
a draw from this distribution.  Then set $X_i = \ind(B_{i} = B_{n+1}).$  The
concurrence graph gives $\probability(X_i = 1) = p_i$, and for $i \neq j$,
$\probability(X_i = X_j) = \probability(B_{i} = B_{j}) = \lambda_{ij}$.

Conversely, suppose such an $n$ dimensional distribution with $\bern(p_i)$
marginals exists.  Let $B_{n+1} \sim \bern(1/2)$ independent of the $X_i$, 
and set 
$B_{i} = B_{n+1} X_i + (1 - B_{n+1})(1 - X_i).$  Then 
$\probability(B_{i} = 1) = (1/2)p_i + (1/2)(1 - p_i) = 1/2$, and 
$\probability(B_{i} = B_{n+1}) = p_i$, the correct concurrence parameter.  Finally, for
$i \neq j$,
\[
\probability(B_{i} = B_{j}) = \probability(X_i = X_j) = \lambda_{ij}.
\]
\end{proof}

Lemma~\ref{LEM:asymmetric} can be used to finish the $n = 4$ case.

\begin{lemma}
A random vector $(B_1,B_2,B_3,B_4)$ with $B_i \sim \bern(1/2)$ exists
(and is possible to simulate in a constant number of steps) 
if and only if for
\begin{align*}
\ell &= \max\{\lambda_{14} + \lambda_{24} + \lambda_{13} + \lambda_{23},
 \lambda_{14} + \lambda_{34} + \lambda_{12} + \lambda_{23},
 \lambda_{24} + \lambda_{34} + \lambda_{12} + \lambda_{13}\} \\
u &= \min_{\{i,j,k\}} (\lambda_{ij} + \lambda_{jk} + \lambda_{ik}),
\end{align*}
it is true that 
\[
\ell \leq u + 1 \text{ and } 1 \leq u.
\]
\end{lemma}

\begin{proof}
By using Lemma~\ref{LEM:asymmetric}, the problem is reduced to finding
a distribution for $(X_1,X_2,X_3)$ where $X_i \sim \bern(\lambda_{i4})$ and
the upper 3 by 3 minor of $\Lambda$ is the new concurrence matrix.  Just as
in Lemma~\ref{LEM:3B}, this gives eight equations of full rank with a 
single parameter $\alpha$.  Letting $q_{ijk} = \prob(X_1 = i,X_2 = j,X_3 = k)$,
the unique solution is 
\begin{align*}
q_{000} &= (1/2)(\lambda_{12} + \lambda_{13} + \lambda_{23}) - (1/2) - \alpha \\
q_{001} &= -(1/2)(\lambda_{14} + \lambda_{24} + \lambda_{13} + \lambda_{23}) + 1 
   + \alpha \\
q_{010} &= -(1/2)(\lambda_{14} + \lambda_{34} + \lambda_{12} + \lambda_{23}) + 1 
   + \alpha \\
q_{011} &= (1/2)(\lambda_{24} + \lambda_{34} + \lambda_{23}) - (1/2) - \alpha \\
q_{100} &= -(1/2)(\lambda_{24} + \lambda_{34} + \lambda_{12} + \lambda_{13}) + 1 
   + \alpha \\
q_{101} &= (1/2)(\lambda_{14} + \lambda_{34} + \lambda_{13}) - (1/2) - \alpha \\
q_{110} &= (1/2)(\lambda_{34} + \lambda_{24} + \lambda_{12}) - (1/2) - \alpha \\
q_{111} &= \alpha.
\end{align*}

All of these right hand sides lie in $[0,1]$ if and only if 
$u \geq 1$, $\ell \leq u + 1$, and $\alpha$ is chosen to lie in
$[(1/2)\ell-1,(1/2)(u - 1)] \cap [0,1]$.
\end{proof}

As with the 3 dimensional case, this proof
can be used to simulate
a 4 dimensional multivariate symmetric Bernoulli:  generate $(X_1,X_2,X_3)$
using any $\alpha \in [(1/2)\ell-1,(1/2)(u - 1)] \cap [0,1]$ and the $q$ distribution, 
then generate
$B_4 \sim \bern(1/2)$, and then set 
$B_i$ to be $B_4 X_i + (1 - B_4)(1 - X_i)$ for $i \in \{1,2,3,4\}$.

\section{Conclusions}

The Fr\'echet-Hoeffding bounds give a lower and upper bound on the pairwise
correlation between two random variables with given marginals.  Hence for
higher dimensions the correlation matrix 
provides edge weights for a convexity graph
whose parameters indicated where on the line from the lower to the upper bound
the correlation lies.  
When it is possible to build a multivariate distribution with these
convexities for marginals that are symmetric Bernoulli, then it is 
possible to build a multivariate distribution with these convexities for
arbitrary marginals.
For two, three and four dimensions, 
the set of convexity matrices that yield a
symmetric Bernoulli distribution is characterized completely.
For five or higher dimensions, every subset of three and four 
have these characterizations as necessary conditions.  

\vskip7mm
\par {\bf Acknowledgement} \hskip3mm Support from the 
National Science Foundation (grant DMS - 1007823) is gratefully acknowledged.

\end{document}